\documentclass[12pt]{amsart}

\usepackage{amsthm, amsmath, amssymb, amsrefs}
\usepackage{mathrsfs}
\usepackage{hyperref}
\usepackage{bbm}
\usepackage{fullpage}

\newcommand{\C}{\mathbb{C}} 
\newcommand{\D}{\mathbb{D}} 

\newcommand{\nat}{\mathbb{N}} 
\newcommand{\HD}{\D_{2}^{\infty}} 

\numberwithin{equation}{section}

\newtheorem{theorem}{Theorem}[section]

\newtheorem{lemma}[theorem]{Lemma}

\theoremstyle{definition}
\newtheorem{definition}[theorem]{Definition}

\newtheorem{remark}[theorem]{Remark}

\begin{document}

\address{Washington University in St. Louis\\ Department of Mathematics\\ One Brookings Drive\\ St. Louis, MO 63130 USA}

\email{geknese@wustl.edu}


\title{The Schur-Agler class in infinitely many variables}

\author{Greg Knese}

\thanks{Partially supported by NSF grant DMS-2247702}


\date{\today}

\begin{abstract}
We define the Schur-Agler class in infinite variables
to consist of functions whose restrictions to finite dimensional
polydisks belong to the Schur-Agler class.  We show 
that a natural generalization of an Agler decomposition 
holds and the functions possess transfer function realizations
that allow us to extend the functions to the unit ball 
of $\ell^\infty$.  We also give a Pick interpolation type theorem
which displays a subtle difference with finitely many variables.
Finally, we make a brief connection to Dirichlet series derived 
from the Schur-Agler class in infinite variables via the Bohr correspondence.
\end{abstract}

\maketitle

\section{Introduction}

This article is about establishing basic properties of
the Schur-Agler class in infinitely many variables.
To back up a bit, the \emph{Schur class} in $N$ variables, $\mathcal{S}_N$,
 will
refer to the set of analytic functions $f:\D^N \to \overline{\D}$
where $\D^N$ is the $N$ dimensional unit polydisk
\[
\D^N= \{z=(z_1,\dots, z_N) \in \C^N: \forall j, |z_j|<1\}.
\]
The Schur class, $\mathcal{S}_{\infty}$, in infinitely many variables will refer
to holomorphic functions (meaning complex Fréchet differentiable)
 on $Ball(\ell^{\infty})$ that
are bounded by one in supremum norm.  
(We review standard notations in Section \ref{sec:notations}.)

A remarkable result, attributed to Hilbert, is
that if we are given Schur class functions $f_N \in \mathcal{S}_N$ for each $N$,
and if for $N>M$ we have
\[
f_N(z_1,\dots, z_M,0,\dots, 0) \equiv f_M(z_1,\dots, z_M),
\]
then
there exists a holomorphic function $f: Ball(c_0)\to \overline{\D}$ 
such that $f(z_1,\dots, z_N,0,\dots) = f_N(z_1,\dots, z_N)$
and such that $f$ is continuous in the norm topology on $Ball(c_0)$.
See \cite{Dirichletbook}, Theorem 2.21 for details.

Even more, $f$ has a homogeneous expansion
\[
f(z) = \sum_{m=0}^{\infty} P_m(z)
\]
where each $P_m$ is an $m$-homogeneous 
form on $c_0$ (see \cite{Dirichletbook}, Proposition 2.28).
Davie-Gamelin \cite{DG} proved that $f$ and its homogeneous
expansion extends further to $Ball(\ell^{\infty})$.
This extension (called the Aron-Berner extension) 
is somewhat elaborate as it requires 
passing to the symmetric $m$-multilinear form
associated to each $P_m$, extending to $\ell^{\infty}$
and then proving that the extended homogeneous expansion
converges in $Ball(\ell^{\infty})$. 

\begin{remark} \label{absremark}
An important subtlety in all of this theory is that, although we can form 
a Taylor series $\sum_{\alpha} f_{\alpha} z^{\alpha}$ associated to $f$
that converges absolutely to $f$ on the Hilbert multidisk $\D_2^{\infty} = Ball(\ell^{\infty})\cap \ell^2$,
in general there will be points in $Ball(c_0)$
where the Taylor series does not converge absolutely.  See \cite{Dirichletbook}, Proposition 4.6 and 
Theorem 10.1. $\diamond$
\end{remark}

\begin{remark} \label{Dirichletremark}
An important motivation in recent decades for the study of $\mathcal{S}_{\infty}$
is through its application to Dirichlet series.  
In particular, the space $H^{\infty}(Ball(c_0))$ of bounded holomorphic functions on $Ball(c_0)$
is isometrically isomorphic to the space $\mathscr{H}^{\infty}$ of Dirchlet series that converge and are bounded on
the right half plane $\{z\in \C: \Re z >0\}$.
The isomorphism, called the the Bohr correspondence, is given by
\[
F \in H^{\infty}(Ball(c_0)) \mapsto f(s) = F(p_1^{-s}, p_2^{-s},\dots) \in \mathscr{H}^{\infty}
\]
where $p_1=2,p_2=3,\dots$ are the prime numbers.
The norm in both cases refers to the supremum norm. 
The space $\mathscr{H}^{\infty}$ and its isomorphism with $H^{\infty}(Ball(c_0))$ 
appears naturally in the study of dilation completeness problems on $L^2(0,1)$ as presented
in Hedenmalm-Lindqvist-Seip \cite{HLS}.  
See also \cite{McCarthy}, \cite{QQ}, \cite{HK}, \cite{Nikolski}, \cite{Nikolskibook},\cite{Dirichletbook}.
$\diamond$
\end{remark}

Returning to the main topic, the Schur-Agler class in $N$ dimensions, $\mathcal{A}_N$, consists of $f \in \mathcal{S}_N$
such that for any $N$-tuple $T=(T_1,\dots, T_N)$ of strictly contractive commuting operators
on a Hilbert space we have 
\[
\|f(T)\| \leq 1
\]
where $f(T)$ is defined using absolutely convergent power series.
We will say \emph{Agler class} for short.  
An inequality of von Neumann \cite{vN} proves that the Agler class in one variable
coincides with the Schur class in one variable; $\mathcal{A}_1=\mathcal{S}_1$.
And\^{o}'s dilation theorem \cite{Ando} proves that the same relation holds in two variables; 
namely; $\mathcal{A}_2=\mathcal{S}_2$.  
Counterexamples first constructed by Varopoulos \cite{varo} show that $\mathcal{A}_N\ne \mathcal{S}_N$
for $N >2$.
A basic motivation for studying the Agler class is that it can provide insights
into the more classical spaces $\mathcal{S}_1,\mathcal{S}_2$---
see Agler-M$^{\text{c}}$Carthy-Young \cite{AMY}, \cite{AMY2}, \cite{AMYbook}.
On the other hand, the Agler class is interesting more broadly:
(1) for studying the operator
theoretic problem of understanding the failure of von Neumann's 
inequality in 3 or more variables
and (2) for providing a large source of interesting and easily
constructible examples of functions within $\mathcal{S}_N$.
Some recent papers on the Agler class are in \cite{Barik}, \cite{Bhowmik},\cite{Debnath}, \cite{Kojin}.

Functions in the Agler class have a variety of useful properties.
First, they possess an \emph{Agler decomposition} and an associated
interpolation theorem.  An Agler decomposition is a formula of the form
\begin{equation} \label{basicad}
1- f(z)\overline{f(w)} = 
\sum_{j=1}^{N}(1-\bar{w}_j z_j) K_j(z,w)
\end{equation}
where $K_1,\dots, K_N$ are positive semi-definite kernels on $\D^N\times\D^N$.
The Agler-Pick interpolation theorem can be stated in the following form.
Some standard references are \cite{Agler}, \cite{AMpick}, \cite{AMbook}.

\begin{theorem}[Agler] \label{AgPick}
Given a finite subset $X \subset \D^N$ and a function $f:X \to \overline{\D}$
the following are equivalent:
\begin{enumerate} 
\item There exists $\tilde{f} \in \mathcal{A}_N$ with $\left.\tilde{f}\right|_{X} = f$.
\item There exist positive semi-definite functions $K_1,\dots, K_N$ on $X\times X$
such that for $z,w \in X$
\[
1- f(z)\overline{f(w)} = \sum_{j=1}^{N} (1-z_j\bar{w}_j)K_j(z,w).
\]

\item For every $N$-tuple $T=(T_1,\dots, T_N)$ of commuting, contractive, simultaneously
diagonalizable matrices whose joint eigenspaces have dimension at most 1 and satisfy
$\sigma(T) \subset X$, we have $\|f(T)\|\leq 1$.
\end{enumerate}

\end{theorem}

Item (1) is a way of phrasing an interpolation problem as an extension problem.
Item (2) is a restriction of an Agler decomposition.
(See Section \ref{sec:notations} for the definition of positive semi-definite function.)
Item (3) says that the function $f$ needs to satisfy a particular type of 
matrix von Neumann
inequality.  Notice that in this case $f(T)$ can be defined by using the diagonalization of $T$, and the dimension
of the space that the $T_j$ act on is at most $\# X$.
Item (3) is not stated explicitly in the literature, at least not in this form, but it is a known component of
the Agler-Pick interpolation theorem.  
(For the skeptical reader, 
the approach in this paper proves a generalization to infinite variables and does not directly rely on the finite variable
theorem so one could pull a proof of the equivalence of (3) by simplifying certain proofs below.)
We think item (3) is important to emphasize since it (conceptually) gives a way 
to \emph{check} if interpolation is possible while item (2) is a useful \emph{conclusion} when you know
interpolation is possible. Item (3) is also the source of a subtlety in infinite variables.

A second key property of Agler class functions is that they possess a contractive transfer function realization formula, which
means the following.  There exists a contractive operator $V$ acting on $\C\oplus \bigoplus_{j=1}^{N} \mathcal{H}_j$
where $\mathcal{H}_1,\dots, \mathcal{H}_N$ are Hilbert spaces such that when we write $V$ in block
form $V = \begin{pmatrix} A & B \\ C & D\end{pmatrix}$ we have
\begin{equation}\label{basictfr}
f(z) = A + B\Delta(z) (1-D\Delta(z))^{-1} C
\end{equation}
where $\Delta(z) = \sum_{j=1}^{N} z_j P_j$ and each $P_j$ represents projection
onto $\mathcal{H}_j$ within the direct sum $\bigoplus_{k=1}^{N} \mathcal{H}_k$.
It turns out that membership in the Agler class can be tested using \emph{generic} matrices.
Namely, analytic $f:\D^N\to \overline{\D}$ belongs to $\mathcal{A}_N$ if for every $N$-tuple $T$
of commuting contractive simultaneously diagonalizable 
matrices with joint eigenspaces having dimension $1$
we have $\|f(T)\|\leq 1$.  With this reduction, defining $f(T)$ only requires
the evaluation of $f$ on the joint eigenvalues of $T$ and not any regularity or
absolute summability.  This can be derived from item (3) in Theorem \ref{AgPick}
or see \cite{polyhedraAMY} (Theorem 6.1 therein) where something more general is proven.

Our goal is to prove that analogues of the Agler decomposition \eqref{basicad}, 
the Agler-Pick interpolation theorem (Theorem \ref{AgPick}), 
and the transfer function realization \eqref{basictfr}
hold in infinitely many variables.  We also wish to make connections
to Dirichlet series as in Remark \ref{Dirichletremark}.
Remark \ref{absremark} suggests that
we cannot define $f(T)$ in the infinite variable setting using absolutely convergent
series.

We would like to remark that this paper is not the first mention of the
Agler class in infinitely many variables---see \cite{AL}, Section 7. 
However one of our larger goals is to start with the simplest definition
possible and deduce basic properties of the Agler class (such as its functional calculus)
 as well as to point out some subtleties of the theory.
 Here is what we imagine to be the simplest definition
 of the Agler class in infinitely many variables.

\begin{definition}
Given a function $f: Ball(c_{00}) \to \C$ we say $f$
is in the \emph{Agler class in infinite variables}, $\mathcal{A}_{\infty}$,
 if for every $N$, the restriction of $f$
to $\D^N$ belongs to the Agler class in $N$ variables, $\mathcal{A}_N$.
\end{definition}

\begin{theorem} \label{mainthm}
If $f \in \mathcal{A}_{\infty}$, then $f$ has a transfer function realization:
there exists a contractive operator $V$ acting on $\C\oplus \bigoplus_{j=1}^{\infty} \mathcal{H}_j$
where $\mathcal{H}_1,\mathcal{H}_2, \dots$ are Hilbert spaces such that when we write $V$ in block
form $V = \begin{pmatrix} A & B \\ C & D\end{pmatrix}$ we have
\begin{equation}\label{tfragain}
f(z) = A + B\Delta(z) (1-D\Delta(z))^{-1} C
\end{equation}
where $\Delta(z) = \sum_{j=1}^{\infty} z_j P_j$ and each $P_j$ represents projection
onto $\mathcal{H}_j$ within the direct sum $\bigoplus_{k=1}^{\infty} \mathcal{H}_k$.
Letting $f_N(z) = f(z_1,\dots, z_N,0 ,\dots)$, we have
the following extension of von Neumann's inequality: 
 for any tuple $T=(T_1,T_2,\dots)$ of
commuting contractive operators (acting on a common Hilbert space)
such that $\sup_{j} \|T_j\| <\infty$ 
we have that
\[
\lim_{N \to \infty} f_N(T)
\]
converges in the strong operator topology to a natural definition of $f(T)$ as
\[
f(T) := (A\otimes I) + (B\otimes I) \Delta(T) (1-(D\otimes I)\Delta(T))^{-1} (C\otimes I)
\]
where $\Delta(T) := \sum_{j=1}^{\infty} P_j \otimes T_j$ (also convergent in 
the strong operator topology).
\end{theorem}

The transfer function formula \eqref{tfragain} for $f$ makes sense for $z \in Ball(\ell^{\infty})$
and we have
\[
\lim_{N\to \infty} f_N(z) = f(z)
\]
for $z \in Ball(\ell^{\infty})$.  
Thus, the transfer function formula readily shows that Agler class functions
extend to $Ball(\ell^{\infty})$.
Our proof relies on a Montel theorem in infinite variables from \cite{Dirichletbook}
and does not use the intricate argument involving nets of
points in $Ball(c_0)$ as in Davie-Gamelin \cite{DG} for the Schur class case.

The transfer function formula \eqref{tfragain}
is essentially equivalent (via a standard argument) to an Agler decomposition 
\[
1- f(z)\overline{f(w)} = 
\sum_{j=1}^{\infty}(1-\bar{w}_j z_j) K_j(z,w).
\]
which we will show converges absolutely.  Again, the $K_j$ are positive semi-definite
kernels on $Ball(\ell^{\infty})\times Ball(\ell^{\infty})$.

\begin{remark} \label{DMremark}
The paper Dritschel-McCullough \cite{DM} discusses a version of the Agler class
in infinite variables via an approach to interpolation and realization formulas using
 \emph{test functions}. 
 Their definition of the Agler class in infinitely many variables
 is more expansive and allows for, for instance, 
 linear functionals on $\ell^{\infty}$ that annihilate $c_0$ and 
 do not have
Agler decompositions in the sense of Theorem \ref{mainthm}.
This expanded Agler class has 
a more general type of Agler decomposition.
See Proposition 5.4 of \cite{DM}. 
Here is a simplified version of an example they present.

Let $\alpha = (a_n)_{n=1}^{\infty}$ be an increasing sequence of
positive real numbers that converge to some $a \in (0,1)$;
e.g. $a_n = \frac{1}{2} - \frac{1}{n+1}$.
By the Hahn-Banach theorem, there exists a linear functional 
$L \in (\ell^{\infty})^*$ such that $L(\alpha) = a$ and $\|L\| = 1$.
We claim $L \notin \mathcal{A}_\infty$.  
If we had an Agler decomposition,
\[
1-L(z)\overline{L(w)} = \sum_{j=1}^{\infty}(1-z_j\bar{w}_j) K_j(z,w)
\]
with each $K_j$ positive semi-definite on $Ball(\ell^{\infty})$
and $\sum_{j=1}^{\infty} K_j(z,z)<\infty$,
then inserting different combinations $z,w \in \{0,\alpha\}$
we have
\[
1 = \sum_{j=1}^{\infty} K_j(0,0), \qquad 
1 = \sum_{j=1}^{\infty} K_j(\alpha,0), \qquad 
1-a^2 = \sum_{j=1}^{\infty} (1-a_j^2)K_j(\alpha,\alpha).
\]
By Cauchy-Schwarz, 
\[
1 \leq \sum_{j=1}^{\infty} |K_j(\alpha,0)| 
\leq \left(\sum_{j=1}^{\infty} K_j(0,0)\right)^{1/2} \left(\sum_{j=1}^{\infty} K_j(\alpha, \alpha)\right)^{1/2} 
\]
so that
\(
1\leq \sum_{j=1}^{\infty} K_j(\alpha, \alpha).
\)
On the other hand, we have
\[
(1-a^2)(1-\sum_{j=1}^{\infty} K_j(\alpha,\alpha)) = \sum_{j=1}^{\infty}(a^2-a_j^2)K_j(\alpha,\alpha) \geq 0
\]
and therefore this must equal zero which would imply $K_j(\alpha,\alpha) = 0$ for all $j$.
This is a contradiction.

Theorem \ref{mainthm} is adapted to functions
that have the added continuity that makes
them completely determined by their values
on $c_{00} \subset c_0$ and so our definition rules out
functions that are zero on all of $c_{00}$.  $\diamond$
\end{remark}

Some aspects of interpolation in $\mathcal{A}_\infty$
are straightforward, however one important aspect
has a subtlety.  

\begin{theorem} \label{Pickthm}
Let $X \subset Ball(\ell^{\infty})$ be a finite subset and
let $f:X \to \overline{\D}$ be a function.
Consider the following conditions.
\begin{enumerate} 
\item There exists $\tilde{f} \in \mathcal{A}_\infty$ with $\left.\tilde{f}\right|_{X} = f$.
\item There exist positive semi-definite functions $K_1, K_2, \dots$ on $X$
such that for $z,w \in X$
\[
1- f(z)\overline{f(w)} = \sum_{j=1}^{\infty} (1-z_j\bar{w}_j)K_j(z,w)
\]
and for all $z\in X$
\[
\sum_{j=1}^{\infty} K_j(z,z) < \infty.
\]
\item For every tuple $T=(T_1,T_2,\dots)$ of commuting, contractive, simultaneously
diagonalizable matrices whose joint eigenspaces have dimension at most 1 and satisfy
$\sigma(T) \subset X$, 
 we have $\|f(T)\|\leq 1$.
\end{enumerate}

Then, 
\begin{itemize}
\item (1) and (2) are equivalent and imply item (3).  
\item Item (3) implies (1) and (2) when $X \subset \HD$.
\end{itemize}
In particular, (1),(2), and (3) are equivalent when $X \subset \HD$.

\end{theorem}

The subtlety alluded to above is that we only
obtain a full generalization of an Agler-Pick interpolation theorem
when our interpolation points lie in $\HD$ and we do not 
know to what extent this condition can be removed.

\begin{remark}
Going back to Remark \ref{DMremark} and the example discussed there, 
the interpolation
problem $0 \in \ell^{\infty} \mapsto 0\in \C, \alpha \mapsto a$ cannot be solved
within $\mathcal{A}_{\infty}$,
however, the associated function $f: \{0,\alpha\} \to \{0,a\}$, $f(0)=0$, $f(\alpha) = a$,
satisfies the condition of item (3) above.
Indeed, if we have commuting simultaneously diagonalizable 
contractive matrices $T_j$
with $\sigma(T_j) \subset \{0,a_j\}$, then $f(T_1,T_2,\dots)$
will be contractive by continuity. Specifically, if $\vec{b}_0$
is the eigenvector for $0$ and $\vec{b}_1$ the eigenvector
for $\alpha$, then contractivity of $T_j$ means
\[
|T_j(\sum_{k=0,1} c_k \vec{b}_k)| \leq |\sum_{k=0,1} c_k \vec{b}_k|
\]
for arbitrary $c_0,c_1\in \C$.
But $T_j(\sum_{k=0,1} c_k \vec{b}_k) = c_1 a_j \vec{b}_1$
and sending $j\to \infty$ we get 
\[
|f(T) (\sum_{k=0,1} c_k \vec{b}_k)| \leq |\sum_{k=0,1} c_k \vec{b}_k|. 
\]

It seems that a complete relaxation of the condition $X\subset \HD$
to $X \subset Ball(\ell^{\infty})$ would lead to the broader notion of Agler 
class constructed with the test function approach of \cite{DM}. 
It would be interesting if the condition $X\subset \HD$ could be
relaxed to $X\subset Ball(c_0)$ with a valid interpolation
theorem in $\mathcal{A}_\infty$.
$\diamond$
\end{remark}

Referring to Remark \ref{Dirichletremark},
it is of interest to understand the image of the map
\[
F \in \mathcal{A}_{\infty} \mapsto f(s) = F(p_1^{-s}, p_2^{-s},\dots) \in \mathscr{H}^{\infty}
\]
into the space $\mathscr{H}^{\infty}$ of convergent and bounded Dirichlet series 
in the right half plane in $\C$.  
We shall let $\mathscr{A}^{\infty}$ denote the image of the above map.
(We caution that as we have defined things
the functions in $\mathscr{A}^{\infty}$ are bounded
by $1$ whereas $\mathscr{H}^{\infty}$ is a Banach space
of functions normed by supremum norm.)
The following is basically a formality but worth pointing out.
Let $\C_+ = \{z\in \C: \Re z >0\}$ denote the right half plane.

\begin{theorem} \label{Dirichletthm}
Let $f \in \mathscr{H}^{\infty}$.
The following are equivalent:
\begin{enumerate}
\item $f\in \mathscr{A}^{\infty}$
\item There exist positive semi-definite kernels $K_1, K_2,\dots$ on $\C_+$
such that 
\[
1- f(s) \overline{f(w)} = \sum_{j=1}^{\infty} (1-p_j^{-(s+\bar{w})})K_j(s,w)
\]
and $\sum_{j=1}^{\infty} K_j(s,s)<\infty$ for each $s\in \C_+$.
\item For every diagonalizable matrix $M$ with $1$ dimensional eigenspaces,
 with $\sigma(M) \subset \C_+$,
and with the property $\|n^{-M}\|\leq 1$ for all $n \in \nat$, we have
\[
\|f(M)\| \leq 1.
\]
\end{enumerate}
\end{theorem} 

Again, $p_1,p_2,\dots$ are the prime numbers.
It would be interesting if the matrices $M$ in item (3)
had a simpler description.  Functions in $\mathscr{A}^{\infty}$
satisfy a special von Neumann inequality.

\begin{theorem} \label{DirichletvNthm}
Let $f \in \mathscr{A}^{\infty}$.  
Suppose $M$ is a bounded operator on a Hilbert space
such that $\sigma(M) \subset \C_{+}$ and $\|n^{-M}\| \leq 1$ for every $n \in \nat$.
Then, 
\[
\|f(M)\| \leq 1
\]
where $n^{-M}$ and $f(M)$ are defined using the Riesz holomorphic functional calculus.
\end{theorem}

The table of contents describes the rest of the paper.

\tableofcontents

\section{Notations and background} \label{sec:notations}

Several spaces of sequences will be of interest.
\begin{itemize}
\item $\nat = \{1,2,\dots\}$.
\item $\ell^{\infty} = \ell^{\infty}(\nat) = \{ z = (z_j)_{j\in \nat} \in \C^{\nat}: \|z\|_{\infty} := \sup_j |z_j| <\infty \}$
\item $\ell^2 = \ell^2(\nat) = \{ z= (z_j)_{j\in \nat} : \sum_{j=1}^{\infty} |z_j|^2 <\infty\}$
\item $c_0 = c_0(\nat) = \{z\in \ell^{\infty}: \lim_{j\to\infty} z_j = 0\}$.
\item $c_{00} = c_{00}(\nat) = \{ z \in \ell^{\infty}: \exists N \in \nat, z_j = 0 \text{ for } j>N\}$. 
\item $Ball(\ell^{\infty}) = \{z\in \ell^{\infty}: \|z\|_{\infty} <1\}$ denotes the open unit ball of $\ell^{\infty}$.
\item $\D^{\infty} = \D^{\nat} = \{z \in \ell^{\infty}: \forall j, |z_j| <1\}$.
\item $Ball(c_0) = \{z\in c_0: \|z\|_{\infty} <1\}$.
\item $Ball(c_{00}) = \{z\in c_{00}: \|z\|_{\infty} < 1\}$.
\item We identify $\D^{N}$ with $\D^N \times \{(0,0,\dots)\} \subset Ball(c_{00})$.
\item The Hilbert multidisk is the set $\D_{2}^{\infty} := \ell^2 \cap Ball(\ell^\infty)$; namely,
the set of sequences $(z_j)_{j\in \nat}$ such that $\sup_j |z_j|<1$ and $\sum_{j} |z_j|^2 <\infty$.
Note that for $z,w \in \D_2^{\infty}$ the infinite product
\[
\prod_{j=1}^{\infty} \frac{1}{1-\bar{w}_j z_j}
\]
converges absolutely.
\item We generally use standard modulus bars $|\cdot|$ for the modulus of complex numbers or vectors (in $\C^N$ or
Hilbert space)
while double bars $\|\cdot\|$ are reserved for operator norms or other norms as listed above. 
\end{itemize}

\begin{remark} \label{psd}
We use the basics of positive semi-definite functions.
Given a set $X$, a function $A:X\times X \to \C$ is positive semi-definite on $X$ if
for every finite subset $Y\subset X$ and every function $a:Y\to \C$ we have
\[
\sum_{z,w \in Y} a(z)\overline{a(w)} A(z,w) \geq 0.
\]
We write $A(z,w) \succeq 0$ in this case.  More generally, we write $A(z,w) \succeq B(z,w)$ if
$A - B \succeq 0$.
We frequently use the Schur product theorem which say that if $A(z,w), B(z,w) \succeq 0$
then $A(z,w)B(z,w) \succeq 0$.
For a function $f:X\to \C$, we let $f\otimes \bar{f}$ denote the function $(z,w) \mapsto f(z)\overline{f(w)}$.
Also, if $X \subset \C^{\nat}$, we let $Z_j \otimes \bar{Z}_j$ denote the function $(z,w) \mapsto z_j\bar{w}_j$.
\end{remark}

\section{Proof of Theorem \ref{mainthm}}

Theorem \ref{mainthm} will be proven with three lemmas.  
The first lemma is our main advance while the other two are
standard.

For the first lemma, 
let $\rho = (\rho_n)_{n\in \nat} \in \HD$ be a fixed square summable sequence
of positive numbers.
Let
$\overline{\D}_{\rho} = \{z =(z_n)_{n\in \nat}: |z_n|\leq \rho_n\}$
and
let $\D_{\rho} = \{z \odot \rho: z \in Ball(c_0)\}$
where
$\rho \odot z = (\rho_jz_j)_{j\in \nat}$.
(The notation is not entirely consistent but it is temporary.)

\begin{lemma} \label{AghasAD}
Assume $f: Ball(c_{00})\to \C$ is in the Agler class, $\mathcal{A}_{\infty}$.

Then, for $j=0,1,2,\dots$, there exist positive semi-definite kernels $K_j$
on $\D_{\rho}$ such that
\[
1- f(z)\overline{f(w)} = K_0(z,w) +
\sum_{j=1}^{\infty}(1-z_j\bar{w}_j) K_j(z,w)
\]
where the sum converges absolutely.
\end{lemma}

Note that we have introduced
the term $K_0$ which is for convenience in the proof.
This term can be absorbed into any of the other terms
for instance as 
\[
(1 - z_1 \bar{w}_1)\left( \frac{K_0(z,w)}{1-z_1\bar{w}_1} + K_1(z,w)\right).
\]
Since $K_0$ will be positive semi-definite and since $\frac{1}{1-z_1\bar{w}_1}$ is
positive semi-definite, the product is positive semi-definite.

\begin{remark}\label{montel}
In the proof, we will use the Montel theorem given in \cite{Dirichletbook} (Theorem 2.17).
It states that for a separable normed vector space $X$, if we are given 
a sequence $(D_n)_{n\in \nat}$ of $D_n \in H^\infty(Ball(X))$ with uniformly bounded
supremum norms, say $\|D_n\|_{\infty} \leq 1$, then there exists
a subsequence $(D_{n_j})_{j\in \nat}$ that converges uniformly on compact subsets 
of $Ball(X)$ to some $D \in H^{\infty}(Ball(X))$ necessarily with $\|D\|_{\infty} \leq 1$.
We will apply this to $X=c_0$ using compact sets of the
form $\overline{\D}_{\rho}$ defined above. $\diamond$
\end{remark}

\begin{proof}[Proof of Lemma \ref{AghasAD}]
For each $N$, let $f_N(z_1,z_2,\dots) = f(z_1,\dots,z_N, 0,\dots)$.
Since $f$ restricted to $\D^N$ has an Agler decomposition,
we can write
\[
1-f_N(z)\overline{f_N(w)} = \sum_{j=1}^{N}(1-\bar{w}_j z_j) K^N_j(z,w)
\]
for positive semi-definite kernels $K_j^N$ that only depend
on the first $N$ variables $z_1,\dots, z_N,w_1,\dots, w_N$, 
are analytic in $z$, and are anti-analytic in $w$.  

Note that for $z,w\in \HD$, the product
\[
S(z,w) = \prod_{j=1}^{\infty} \frac{1}{1-\bar{w}_j z_j}
\]
converges and is positive semi-definite.
Note that
\[
S_n(z,w) = (1-\bar{w}_n z_n) S(z,w) = \prod_{j\ne n} \frac{1}{1-\bar{w}_j z_j}
\]
is also positive semi-definite
and $S, S_n \succeq 1$ using the partial order from Remark \ref{psd}.
Here `$1$' is the identically $1$ function.

Let $f_N\otimes \overline{f_N}$ denote the function
$(z,w) \mapsto f_N(z)\overline{f_N(w)}$.
We have the positive semi-definite function inequality
\begin{equation} \label{psdineq}
S \succeq (1-f_N\otimes \overline{f_N})S = \sum_{j=1}^{N} K_j^{N} S_j
\succeq \sum_{j=1}^{N} K_j^{N} \succeq \sum_{j=1}^{M} K_j^N
\end{equation}
for $M<N$
and by Cauchy-Schwarz we have for $z,w\in \HD$
\[
S(z,z)^{1/2}S(w,w)^{1/2} \geq |K_j^N(z,w)|.
\]
Recall $\rho \odot z = (\rho_jz_j)_{j\in \nat}$.
For $z,w \in Ball(c_0)$, $K_j^N(\rho\odot z, \rho \odot \bar{w})$
is bounded and analytic in $Ball(c_0)\times Ball(c_0)$ with supremum norm
bounded by $S(\rho,\rho)$.  By the Montel
theorem (see Remark \ref{montel}), 
for each $j=1,2,\dots$ in succession
 there is a subsequence of $N\in \nat$ 
such that 
\begin{equation} \label{converge} 
K_j^N(\rho\odot z, \rho \odot \bar{w}) \overset{N\to\infty}{\to} K_j(\rho\odot z, \rho \odot \bar{w})
\text{ uniformly on compact subsets of } Ball(c_0).
\end{equation}
The limiting functions, $K_j$, are necessarily positive semi-definite because this property
is preserved under limits.
By a standard diagonal argument, we can find a common subsequence of $N$
such that for all $j$, \eqref{converge} holds.
Recall $\D_{\rho} = \{z \odot \rho: z \in Ball(c_0)\}$.
By \eqref{psdineq}, for $z,w\in \D_{\rho}$, 
\(
S\succeq \sum_{j=1}^{M} K_j^N
\)
for $M<N$ and sending $N \to \infty$ we have
$S \succeq \sum_{j=1}^{M} K_j$.
Finally, we can send $M\to \infty$ to obtain
\(
S \succeq \sum_{j=1}^{\infty} K_j
\)
with absolute convergence.
Absolute convergence can be proven by looking at the diagonal $z=w$ first and then
applying Cauchy-Schwarz.
To finish the proof, we show
\[
K_0(z,w) := 1-f(z)\overline{f(w)} - \sum_{j=1}^{\infty}(1-z_j \bar{w}_j)K_j(z,w)
\]
is positive semi-definite.  Here we emphasize that $f$ on $\D_{\rho}$ is defined 
via $f$'s holomorphic extension from $Ball(c_{00})$ to $Ball(c_0)$ as in the
traditional Schur class of infinitely many variables.

Now, for $N<M$
\[
1-f_M(z)\overline{f_M(w)} - \sum_{j=1}^{N}(1-z_j \bar{w}_j)K_j^{M}(z,w)
= 
\sum_{j=N+1}^{M}(1-z_j\bar{w}_j)K_j^{M}(z,w)\\
\]
Let $Z_j \otimes \overline{Z_j}$ denote the function $(z,w) \mapsto z_j\bar{w}_j$.
Since $S \succeq K_j^{M}$ we see that
\[
1 - f_M\otimes \overline{f_M} - \sum_{j=1}^{N} (1 - Z_j \otimes \overline{Z_j}) K_j^M
\succeq - \sum_{j=N+1}^{M} (Z_j \otimes \overline{Z_j}) S.
\]
This inequality means that for any finite subset $Y \subset \D_{\rho}$
and any function $a:Y \to \C$
\[
\sum_{z,w \in Y} (1-f_M(z)\overline{f_M(w)} - \sum_{j=1}^{N}(1-z_j \bar{w}_j)K_j^{M}(z,w))a(z) \overline{a(w)}
\geq
-\sum_{z,w\in Y} \sum_{j=N+1}^{M} z_j\bar{w}_j S(z,w) a(z) \overline{a(w)}.
\]
(See Remark \ref{psd}.)
Now,
\[
\sum_{z,w\in Y} \sum_{j=N+1}^{M} z_j\bar{w}_j S(z,w) a(z) \overline{a(w)}
\leq
S(\rho,\rho) \sum_{j=N+1}^{M}\rho_j^2 \left(\sum_{z\in Y} |a(z)|\right)^2
\leq
S(\rho,\rho) \sum_{j=N+1}^{\infty}\rho_j^2 \left(\sum_{z\in Y} |a(z)|\right)^2.
\]
Sending $M\to \infty$
\[
\sum_{z,w \in Y} (1-f(z)\overline{f(w)} - \sum_{j=1}^{N}(1-z_j \bar{w}_j)K_j(z,w))a(z) \overline{a(w)}
\geq 
-S(\rho,\rho) \sum_{j=N+1}^{\infty}\rho_j^2 \left(\sum_{z\in Y} |a(z)|\right)^2
\]
and then sending $N\to \infty$
\[
\sum_{z,w \in Y} (1-f(z)\overline{f(w)} - \sum_{j=1}^{\infty}(1-z_j \bar{w}_j)K_j(z,w))a(z) \overline{a(w)} \geq 0
\]
which proves $K_0(z,w)$ 
is positive semi-definite.
\end{proof}

\begin{lemma} \label{lemmalurking}
Let $X \subset Ball(\ell^{\infty})$.
Assume $f: X\to \C$ is a function such that 
for $j=0,1,2,\dots$, there exist positive semi-definite kernels $K_j:X\times X \to \C$
on $X$ such that
\[
1- f(z)\overline{f(w)} = 
\sum_{j=1}^{\infty}(1-z_j\bar{w}_j) K_j(z,w)
\]
where the sum converges absolutely.
Then, 
there exists a contractive operator $V$ acting on $\C\oplus \bigoplus_{j=1}^{\infty} \mathcal{H}_j$
where $\mathcal{H}_1,\mathcal{H}_2, \dots$ are Hilbert spaces such that when we write $V$ in block
form $V = \begin{pmatrix} A & B \\ C & D\end{pmatrix}$ we have
\[
f(z) = A + B\Delta(z) (1-D\Delta(z))^{-1} C
\]
\[
K_j(z,w) =  C^*(I-\Delta(w)^* D^*)^{-1}P_j(I-D\Delta(z))^{-1} C
\]
where $\Delta(z) = \sum_{j=1}^{\infty} z_j P_j$ and each $P_j$ represents projection
onto $\mathcal{H}_j$ within the direct sum $\bigoplus_{k=1}^{\infty} \mathcal{H}_k$.
With these formulas, $f$ and $K_j$ extend to $Ball(\ell^{\infty})$ and
there exists a positive semi-definite kernel $K_0(z,w)$ such that
\[
1- f(z)\overline{f(w)} = K_0(z,w) +
\sum_{j=1}^{\infty}(1-z_j\bar{w}_j) K_j(z,w)
\]
holds on $Ball(\ell^{\infty})\times Ball(\ell^{\infty})$.

\end{lemma}

The proof is a standard lurking isometry argument (for those who know what that is)
that we include for completeness (for those who do not).

\begin{proof}
By the Moore-Aronszajn theorem on reproducing kernel Hilbert space,
we can factor $K_j(z,w) = K_{j,z}^* K_{j,w}$ for $K_{j,z}$ an element
of some Hilbert space $\mathcal{H}_j$.
We write $K_{j,z}^* K_{j,w}$ instead of $\langle K_{j,w}, K_{j,z} \rangle$ 
and view $K_{j,z}^*$ as an element of the dual $\mathcal{H}_j^*$ of $\mathcal{H}_j$.  
The following map 
\[
\begin{pmatrix} 1 \\ z_1(K_{1,z})^* \\ z_2 (K_{2,z})^* \\ \vdots
\end{pmatrix}
\mapsto
\begin{pmatrix} f(z) \\ (K_{1,z})^* \\ (K_{2,z})^* \\ \vdots \end{pmatrix}
\]
initially defined for vectors indexed by $z\in X$
extends linearly and in a well-defined way to a contractive operator $V$
from $\C \oplus \bigoplus_{j=1}^{\infty} \mathcal{H}_j^*$
to $\C \oplus \bigoplus_{j=1}^{\infty} \mathcal{H}_j^* $.
We write $V$ in block form
\(
\begin{pmatrix} A & B \\
C & D \\
\end{pmatrix}
\)
with 
$A \in \C \cong \mathcal{B}(\C,\C)$, 
$B \in \mathcal{B}(\bigoplus_{j=1}^{\infty} \mathcal{H}_j^*, \C)$,
$C \in \mathcal{B}(\C, \bigoplus_{j=1}^{\infty} \mathcal{H}_j^*)$,
$D \in \mathcal{B}(\bigoplus_{j=1}^{\infty} \mathcal{H}_j^*)$,
using the notation $\mathcal{B}(X,Y)$ to denote the bounded
linear operators from $X$ to $Y$ (as well as $\mathcal{B}(X)= \mathcal{B}(X,X)$).
Let $\Delta(z) \in \mathcal{B}(\bigoplus_{j=1}^{\infty} \mathcal{H}_j^*)$ be the diagonal operator
sending 
\[
(h_j)_{j\in \nat} \in \bigoplus_{j=1}^{\infty} \mathcal{H}_j^* \mapsto (z_j h_j)_{j\in\nat} \in \bigoplus_{j=1}^{\infty} \mathcal{H}_j^*.
\]
Let $F(z) := (K_{j,z}^*)_{j\in \nat} \in \bigoplus_{j=1}^{\infty} \mathcal{H}_j^*$. 
Then,
\[
V \begin{pmatrix} 1 \\ \Delta(z) F(z) \end{pmatrix} = \begin{pmatrix} f(z) \\ F(z) \end{pmatrix}
\]
which implies

\begin{align} \label{blocksolve1}
A + B \Delta(z) F(z) &= f(z) \\ \label{blocksolve2}
C+ D \Delta(z) F(z) &= F(z) 
\end{align}

 Solving for $F(z)$ and then $f(z)$ we obtain
 \[
 \begin{aligned}
 F(z) &= (I-D\Delta(z))^{-1} C \\
 f(z) &= A + B \Delta(z) (I-D \Delta(z))^{-1} C.
 \end{aligned}
 \]
 Note the expressions on the right are defined as written
 for $z \in Ball(\ell^{\infty})$.
 Evidently, $F(w)^*P_jF(z)$ extends $K_j(z,w)$.
 Since $V$ is contractive, we can factor $I-V^*V = W^*W$
 for some operator $W$.
 Let 
 \[
 G(z) = W  \begin{pmatrix} 1 \\ \Delta(z) F(z) \end{pmatrix}
 \]
 so that
 \[
 \begin{pmatrix} 1 \\ \Delta(w) F(w) \end{pmatrix}^* \begin{pmatrix} 1 \\ \Delta(z) F(z) \end{pmatrix}
 = \begin{pmatrix} f(w) \\ F(w) \end{pmatrix}^* \begin{pmatrix} f(z) \\ F(z) \end{pmatrix}
 + G(w)^*G(z).
 \]
 This rearranges into
 \[
 1- f(z)\overline{f(w)} =
 F(w)^*(I-\Delta(w)^*\Delta(z))F(z) + G(w)^*G(z)
 \]
 and since
 \[
 F(w)^*(I-\Delta(w)^*\Delta(z))F(z)
 =
 \sum_{j=1}^{\infty} (1-\bar{w}_jz_j) F(w)^*P_j F(z)
 \]
 we have the desired extension of the Agler decomposition
 using $K_0(z,w) = G(w)^*G(z)$.
    \end{proof}

    Lemmas \ref{AghasAD} and \ref{lemmalurking} establish most of Theorem \ref{mainthm}.
    For the remaining part, we need a basic estimate on transfer function formulas
    in order to establish the full von Neumann inequality for the Agler class.
 
 \begin{lemma} \label{tfrest}
Suppose $V$ is a contractive operator acting on $\C\oplus \bigoplus_{j=1}^{\infty} \mathcal{H}_j$,
where $\mathcal{H}_1,\mathcal{H}_2, \dots$ are Hilbert spaces.
Writing $V$ in block
form $V = \begin{pmatrix} A & B \\ C & D\end{pmatrix}$ we define
for $z \in Ball(\ell^{\infty})$
\[
f(z) = A + B\Delta(z) (1-D\Delta(z))^{-1} C
\]
where $\Delta(z) = \sum_{j=1}^{\infty} z_j P_j$ and each $P_j$ represents projection
onto $\mathcal{H}_j$ within the direct sum $\bigoplus_{k=1}^{\infty} \mathcal{H}_k$.
 Then, for $z,w \in Ball(\ell^{\infty})$
 \[
 f(z) - f(w) = B(1-\Delta(z)D)^{-1}\Delta(z-w)(1-D\Delta(w))^{-1}C.
 \]
 
 \end{lemma}
 
 \begin{proof}
 \[
\begin{aligned}
f(z) - f(w) &= B(\Delta(z)(1-D\Delta(z))^{-1} - \Delta(w)(1-D\Delta(w))^{-1})C \\
&= B( \Delta(z)((1-D\Delta(z))^{-1} - (1-D\Delta(w))^{-1}) + \Delta(z-w) (1-D\Delta(w))^{-1})C\\
&=B(\Delta(z)(1-D\Delta(z))^{-1})(D\Delta(z-w))(1-D\Delta(w))^{-1} + \Delta(z-w)(1-D\Delta(w))^{-1})C\\
&= B(\Delta(z)(1-D\Delta(z))^{-1})D + 1)\Delta(z-w)(1-D\Delta(w))^{-1})C\\
&=
B(1-\Delta(z)D)^{-1}\Delta(z-w)(1-D\Delta(w))^{-1}C.
\end{aligned}
\]
\end{proof}
 
 Now we finish the proof of Theorem \ref{mainthm}.
As written in Theorem \ref{mainthm}, for an infinite tuple $T = (T_1,T_2,\dots)$
of operators on a common Hilbert space $\mathcal{H}$ satisfying
\[
\|T\|_{\infty} := \sup_j \|T_j\| < 1
\]
we define
\[
f(T) := (A\otimes I) + (B\otimes I) \Delta(T) (1-(D\otimes I)\Delta(T))^{-1} (C\otimes I)
\]
where $\Delta(T) := \sum_{j=1}^{\infty} P_j \otimes T_j$.  Note that this definition
does not require the operators $(T_j)_j$ to pairwise commute but if they do then
each $T_j$ commutes with $f(T)$ since each $I\otimes T_j$ commutes with $\Delta(T)$.

Since $(D\otimes I)\Delta(T)$ is strictly contractive, $f(T)$ equals the
absolutely convergent sum
\[
(A\otimes I) + (B\otimes I) \Delta(T)\sum_{j=0}^{\infty} ((D\otimes I)\Delta(T))^{j} (C\otimes I)
\]
and substituting $z\in Ball(\ell^{\infty})$ we obtain an absolutely convergent
homogeneous expansion for $f(z)$
\begin{equation} \label{tfrsum}
f(z) = A + B\Delta(z)\sum_{j=0}^{\infty} ((D\Delta(z))^{j} C.
\end{equation}

Before we finish the proof of Theorem \ref{mainthm} we make 
some clarifications about our functional calculus.

\begin{remark}
In finitely many variables we stated that our convention/definition 
for $f(T)$ is via an absolutely convergent power series expansion.
Therefore, it should be pointed out that this new formulation
of ``$f(T)$'' using the transfer function formula matches the old one.
All that really needs to be said is that when we
insert $z=(z_1,\dots, z_N,0,\dots)$ into \eqref{tfrsum}
the homogeneous terms $B\Delta(z)(D\Delta(z))^jC$
are homogeneous polynomials and since $f$ is
analytic on $\D^N$, the monomial sum we obtain
from expanding $B\Delta(z)(D\Delta(z))^jC$
is absolutely convergent in $\D^N$.  Thus,
evaluating $f(T)$ at a finite tuple $T=(T_1,T_2,\dots, T_N,0,\dots)$
can either be evaluated using the transfer function formula
or the absolutely convergent power series.  $\diamond$ \end{remark}

The proof of Lemma \ref{tfrest} extends directly to prove that for another such tuple $S$
acting on the same Hilbert space $\mathcal{H}$ as $T$ 
we have
\[
f(T) - f(S) = (B\otimes I)(1-\Delta(T)(D\otimes I))^{-1}\Delta(T-S)(1-(D\otimes I)\Delta(S))^{-1}(C\otimes I).
\]
This implies the estimate that for $x \in \mathcal{H}$
\[
|(f(T)-f(S))x|^2 \leq \|(B\otimes I)(1-\Delta(T)(D\otimes I))^{-1}\|^2 |\Delta(T-S)(1-(D\otimes I)\Delta(S))^{-1}(C\otimes x)|^2.
\]
Letting $T^{(N)} = (T_1,\dots, T_N,0,\dots)$ we have $f(T^{(N)}) = f_N(T)$ and
\[
\begin{aligned}
|(f_N(T)-f(T))x|^2 &\leq (1-\|T\|_{\infty})^{-2} |\Delta(T^{(N)}-T)(1-(D\otimes I)\Delta(T))^{-1}(C\otimes x)|^2\\
&=
(1-\|T\|_{\infty})^{-2} \sum_{j=N+1}^{\infty} |(P_j\otimes T_j)(1-(D\otimes I)\Delta(T))^{-1}(C\otimes x)|^2
\end{aligned}
\]
and this goes to $0$ as $N\to \infty$.  Thus, $f_N(T) \to f(T)$ in the strong operator topology.
This concludes the proof of Theorem \ref{mainthm}.

\section{Proof of Theorem \ref{Pickthm}}
That (1) implies (2) follows from Lemmas \ref{AghasAD} and \ref{lemmalurking}.
Proving (2) implies (1) is a standard lurking isometry argument 
that is somewhat similar to our proof of Theorem \ref{mainthm}.  
Proving (1) implies (3) consists mostly of technicalities that we discuss next.
The proof of (3) implies (1) is the main contribution below.

Regarding (1) implies (3), by definition,
 we can make sense of $\tilde{f}(T)$
for $\tilde{f} \in \mathcal{A}_\infty$ when $(\|T_j\|)_{j\in \nat} \in Ball(c_{00})$
and Theorem \ref{mainthm} lets us make sense of it when
$(\|T_j\|)_{j\in \nat} \in Ball(\ell^{\infty})$.  
To prove $\|\tilde{f}(T)\|\leq 1$ when $T$ consists of \emph{matrices}
that are commuting, contractive, simultaneously diagonalizable
with joint spectrum $\sigma(T) \subset X$ and eigenspaces of dimension
at most $1$, we can give a continuity argument.
First, $\|\tilde{f}(rT)\|\leq 1$ holds for $r<1$ because we will have 
\[
\tilde{f}_N(rT) \to \tilde{f}(rT)
\]
in the strong operator topology as $N\to \infty$---as in previous
sections $\tilde{f}_N$ refers to the restriction of $\tilde{f}$ to $\D^N$.
Note that $\tilde{f}_N(rT)$ is defined in terms of the absolutely convergent
power series of $\tilde{f}_N$.  
Next, we use the diagonalizability properties of $T$;
let $\vec{b}(z)$ be the eigenvector associated to joint eigenvalue $z\in \sigma(T)$.
Then, for any function $a: \sigma(T) \to \C$ we have
\[
| \tilde{f}(rT) \sum_{z\in \sigma(T)} a(z) \vec{b}(z)| = |\sum_{z\in \sigma(T)} \tilde{f}(rz) a(z) \vec{b}(z)|
\leq |\sum_{z\in \sigma(T)} a(z) \vec{b}(z)|.
\]
We can send $r\nearrow 1$ to conclude $\|\tilde{f}(T)\| \leq 1$.  
This proves (1) implies (3).  

Our main contribution is the proof of (3) implies (2) assuming the finite set $X$ belongs to $\HD$.
This is a modification of the finite variable
cone separation argument; the main difference being
Lemma \ref{closedlemma} below.
Consider the following cone of functions on $X\times X$
\[
\begin{aligned}
\mathcal{C} = \left\{ (z,w) \right.&\mapsto \sum_{j=1}^{\infty} (1-z_j\bar{w}_j) A_j(z,w): \\
& A_1, A_2,\dots \text{ are positive semi-definite functions on } X; \\
&\left. \forall z\in X, \sum_{j=1}^{\infty} A_j(z,z)<\infty\right\}.
\end{aligned}
\]

\begin{lemma} \label{closedlemma}
$\mathcal{C}$ is closed.
\end{lemma}
\begin{proof}
Let 
\[
C_n(z,w) = \sum_{j=1}^{\infty} (1-z_j\bar{w}_j) A_{n,j}(z,w)
\]
define a sequence of functions in $\mathcal{C}$ that converges to the function $C:X\times X\to \C$ 
pointwise; in particular, each $A_{n,j}$ is positive semi-definite.
  We must show $C \in \mathcal{C}$.

Let $\delta:X\times X \to \C$ denote the function $\delta(z,w) = 1$ if $z=w$ and $\delta(z,w) = 0$ if $z\ne w$.
There necessarily exist constants $c_1, c_2$ such that
for all $n$,
\(
c_1 \delta \succeq C_n \succeq c_2 \delta.
\)
This looks more familiar when we view our functions on $X\times X$
as matrices.
Since $X \subset \D_2^{\infty}$ we can define
\[
S(z,w) = \prod_{j=1}^{\infty} \frac{1}{1-z_j \bar{w}_j} \text{ and } S_j(z,w) = (1-z_j\bar{w}_j)S(z,w) 
= \prod_{i\ne j} \frac{1}{1-z_i \bar{w}_i};
\]
which are positive semi-definite and satisfy $S, S_j \succeq 1$, with `$1$' representing the identically $1$
function.
Some parts of what follow are similar to the proof of Lemma \ref{AghasAD}.
Now, for any $i \in \nat$
\[
c_1S \succeq C_n S = 
\sum_{j=1}^{\infty} S_j A_{n,j} 
\succeq  \sum_{j=1}^{\infty} A_{n,j}
\succeq A_{n,i} \succeq 0
\]
which shows the matrices $A_{n,i}$ are uniformly bounded. Let $c_3>0$ satisfy $c_3 \delta \succeq c_1 S$.

Using a diagonal argument we can select a subsequence such that each $A_{n,i}$
converges to a positive semi-definite matrix $A_i$ as $n\to \infty$.  
Also, for each $N$ we have
\(
C_n S \succeq \sum_{j=1}^{N} S_j  A_{n,j} \succeq \sum_{j=1}^{N} A_{n,j}.
\)
Sending $n \to \infty$ we have
\(
C S \succeq \sum_{j=1}^{N} S_j A_j \succeq \sum_{j=1}^{N} A_j
\)
and sending $N\to \infty$ we have
\(
C S \succeq 
\sum_{j=1}^{\infty} S_j  A_j \succeq
\sum_{j=1}^{\infty} A_j
\)
where the sums converge absolutely.
Next, recalling $Z_j\otimes \bar{Z}_j$ denotes the
function $(z,w) \mapsto z_j\bar{w}_j$ we have
\[
\begin{aligned}
C_n &- \sum_{j=1}^{N} (1-Z_{j}\otimes \bar{Z}_j)A_{n,j} \\
&= 
\sum_{j=N+1}^{\infty} (1-Z_j\otimes \bar{Z}_j)A_{n,j}\\
&\succeq
-c_3 \sum_{j=N+1}^{\infty}  (Z_{j}\otimes \bar{Z}_j) \delta\\
&\succeq
-c_3  \max\{\sum_{j=N+1}^{\infty} |z_j|^2: z \in X\} \delta
\end{aligned}
\]
The last inequality amounts to the fact that for any function $a:X\to \C$
we have
\[
\begin{aligned}
&\sum_{z,w\in X} \sum_{j=N+1}^{\infty} (z_j\bar{w}_j) \delta(z,w) a(z) \overline{a(w)} \\
&=
\sum_{z\in X}\sum_{j=N+1}^{\infty} |z_j|^2 |a(z)|^2 \\
&\leq
\max\{\sum_{j=N+1}^{\infty} |z_j|^2: z\in X\} \sum_{z\in X}|a(z)|^2. 
\end{aligned}
\]
Setting $M_N = \max\{\sum_{j=N+1}^{\infty} |z_j|^2: z\in X\}$ we have
$M_N \to 0$ since $X \subset \ell^2$.
Sending $n \to \infty$ we have
\[
C - \sum_{j=1}^{N} (1-Z_j\otimes \bar{Z}_j)A_j \succeq -c_3 M_N \delta
\]
and finally sending $N\to \infty$ we have
\[
A_0 := C - \sum_{j=1}^{\infty} (1-Z_j\otimes \bar{Z}_j)A_j \succeq 0.
\]
Finally, $A_0$ can be absorbed into any other term, say $A_1$ as 
$\tilde{A}_1 = A_1 + \frac{1}{1-Z_1\otimes \bar{Z}_1} A_0$
to see that 
\[
C =  (1-Z_1\otimes \bar{Z}_1)\tilde{A}_1+ \sum_{j=2}^{\infty} (1-Z_j\otimes \bar{Z}_j)A_j
\]
is of the desired form.
\end{proof}

What follows is now standard.
Suppose $f:X \to \overline{\D}$ is a function
with the assumed property in item (3).
Suppose the function on $X\times X$, $F= 1 - f\otimes \bar{f}$ is not in the cone $\mathcal{C}$.
By the Hahn-Banach hyperplane separation theorem, there exists a 
function $B: X\times X \to \C$ with $B(z,w) = \overline{B(w,z)}$ such that 
\[
\sum_{z,w \in X} F(z,w) B(z,w) <0 \text{ and for all } C \in \mathcal{C} 
\text{ we have } \sum_{z,w\in X} C(z,w) B(z,w) \geq 0.
\]

The second condition implies $B \succeq 0$ by setting 
$C = (1-Z_1\otimes \bar{Z}_1) \frac{ a\otimes \bar{a} }{1-Z_1\otimes \bar{Z}_1}$
for an arbitrary function $a:X\to \C$
and observing
\[
 \sum_{z,w\in X} C(z,w) B(z,w) = \sum_{z,w} a(z)\overline{a(w)} B(z,w) \geq 0.
\]
 Next, we factor $B(w,z) = \overline{B(z,w)} = \vec{b}(z)^* \vec{b}(w)$ for vectors $\vec{b}(z) \in \C^r$ where
 $r$ is the rank of the matrix$(B(z,w))_{z,w\in X}$.  
 
 Choosing now $C_j = (1-Z_j\otimes \bar{Z}_j) (a\otimes \bar{a})$ for an arbitrary $a:X\to \C$ 
 \[
 \begin{aligned}
 0&\leq \sum_{z,w \in X} (1-z_j\bar{w}_j)a(z) \overline{a(w)}B(z,w)\\
 &=  \sum_{z,w\in X} (1-z_j\bar{w}_j) a(z) \overline{a(w)} \vec{b}(z)^t \overline{\vec{b}(w)}\\
 &=
 \left|\sum_{w\in X} a(w) \vec{b}(w)\right|^2 - \left|\sum_{w\in X} w_j a(w) \vec{b}(w)\right|^2 \\
 \end{aligned}
 \]
 which proves that maps $T_j: \sum_{w\in X} a(w) \vec{b}(w) \mapsto \sum_{w\in X} w_j a(w) \vec{b}(w)$
 are well-defined, contractive, diagonalizable.  
Setting $T=(T_1,T_2,\dots)$, we have assumed that $f(T)$ is contractive
which means for all functions $a:X\to \C$
\[
\left|\sum_{w\in X} a(w) \vec{b}(w)\right|^2 - \left|\sum_{w\in X} f(w) a(w) \vec{b}(w)\right|^2 \geq 0
\] 
 and this means
 \[
 \sum_{z,w\in X} (1-f(z)\overline{f(w)}) B(z,w) a(z) \overline{a(w)} \geq 0
 \]
 and this means
 \[
FB = (1- f\otimes \bar{f})B \succeq 0
 \]
 (recall $F = 1-f\otimes f$)
and in particular
\[
\sum_{z,w\in X} F(z,w) B(z,w) \geq 0
\]
which is a contradiction. 
This proves $F \in \mathcal{C}$ or more precisely,
there exist positive semi-definite matrices $A_1, A_2,\dots$ with
rows and columns indexed by $S$ such that
\[
1-f(z)\overline{f(w)} = \sum_{j=1}^{\infty} (1-z_j\bar{w}_j) A_j(z,w)
\]
and $\sum_{j=1}^{\infty} A_j(z,z) < \infty$.
This proves that (2) implies (3) as well as the full proof of Theorem \ref{Pickthm}.

\section{Proof of Theorem \ref{Dirichletthm}}

The equivalence of (1) and (2) is straightforward based
on previous results.
The implication (1) implies (3) follows from Theorem \ref{Pickthm}.

 To prove (3) implies (2), 
 we note that our function $f:\C_+\to \overline{\D}$ gives rise to a function $F: X \to \overline{\D}$
where 
\[
X = \{\pi(s) := (p_1^{-s},p_2^{-s},\dots): s\in \C_+\} \subset Ball(c_0)
\]
and $F(\pi(s)) = f(s)$.
We can apply the interpolation result of Theorem \ref{Pickthm} as follows. 
Let $Y \subset \{s \in \C: \Re s >1/2\}$ be a countable set with a limit point in $\{s \in \C: \Re s >1/2\}$.
We want a limit point so that it is a determining set for holomorphic functions and we want
$\Re s > 1/2$ so that $\pi(s) \in \HD$.  
Let $\bigcup_{n=1}^{\infty} Y_n = Y$ be an increasing union of finite sets.
Set $X_n = \pi(Y_n), X_{\infty} = \pi(Y)$.     

Fix $n$ and 
let $T = (T_1, T_2,\dots)$ be an infinite tuple of contractive, commuting matrices
with $\sigma(T) \subset X_n$ and 1 dimensional joint eigenspaces.  
The eigenvalues of $T_1$ are of the form $p_1^{-s} = 2^{-s}$ for $s \in Y_n$.

We can take $M = -\log T_1/\log 2$ using the principal log and then $T_j = p_j^{-M}$
for all $j$.  By assumption, $\|T_j\|= \|p_j^{-M}\| \leq 1$ and this extends to
all natural numbers by factoring into primes: $\|n^{-M}\| \leq 1$ for all $n\in \nat$.
By assumption (3), $\|f(M)\| = \|F(T)\| \leq 1$.  
Therefore, by the implication (3) $\implies$ (1) in Theorem \ref{Pickthm},
there exists $G_n \in \mathcal{A}_{\infty}$ such that $\left.G_n\right|_{X_n} = \left.F\right|_{X_n}$.
In particular, $g_n(s) := G_n(\pi(s)) \in \mathscr{A}^{\infty}$
and agrees with $f_n$ on $Y_n$.  
 
By the Montel Theorem of Remark \ref{montel},
since $G_n \in \mathcal{A}_{\infty} \subset \mathcal{S}_{\infty}$,
there is a subsequence of $n\in \nat$ such that $G_n \to G \in \mathcal{S}_{\infty}$ 
uniformly on compact subsets of $Ball(c_0)$.  
The limit $G$ necessarily agrees with $F$ on $X_{\infty}$.
Since membership in $\mathcal{A}_{\infty}$ can be tested by
checking whether $G$ restricted to $\D^N$ belongs to $\mathcal{A}_N$ by Theorem \ref{mainthm},
and since this in turn can be tested by examining $G$ on finite subsets
of $\D^N$, we see that the limit $G$ belongs to $\mathcal{A}_{\infty}$
since each $G_n$ belongs to $\mathcal{A}_{\infty}$.
Finally, $g = G\circ \pi \in \mathscr{A}^{\infty}$ and agrees
with $f$ on the set of uniqueness $Y$, so we see that $g=f \in \mathscr{A}^{\infty}$.

\section{Proof of Theorem \ref{DirichletvNthm}}

Given $f \in \mathscr{A}^{\infty}$, let $F \in \mathcal{A}_{\infty}$
be the Bohr lift of $f$; namely, $f(s) = F(p_1^{-s},p_2^{-s},\dots)$.
Again, forming $F_N(z_1,z_2,\dots) = F(z_1,\dots, z_N,0,\dots)$,
we see that the functions $f_N(s) := F_N(p_1^{-s},p_2^{-s},\dots)$
converge uniformly to $f$ on the half planes 
$\{s\in \C: \Re s \geq \epsilon\}$ for each $\epsilon>0$.
The power series for $F_N$ converges absolutely on $\D^N$
and therefore the Dirichlet series for 
$f_N$ converges to $f_N$ absolutely on $\C_{+}$.
Because of this the holomorphic
functional calculus evaluation $f_N(M)$
agrees with $F_N(p_1^{-M},p_2^{-M},\dots)$.  
Since $F_N \in \mathcal{A}_{\infty}$, 
\[
\|f_N(M)\| = \|F_N(p_1^{-M},p_2^{-M},\dots)\| \leq 1.
\]
Since $f_N$ converges uniformly to $f$ on a half-plane
containing $\sigma(M)$, $f_N(M) \to f(M)$ in operator norm
and therefore $\|f(M)\| \leq 1$ as desired. This
completes the proof.
 
\section*{Acknowledgments}
Thank you to the referee for a thoughtful report and for pointing out several typos.
Thank you to Professor Daniel Alpay for bringing the reference \cite{AL} to our attention.

\end{document}